\documentclass[12pt]{article}
\usepackage{mathrsfs}

\usepackage{multirow}
\usepackage{booktabs}  %±í¸ñ
\usepackage{t1enc}
\usepackage[latin1]{inputenc}
\usepackage[english]{babel}
\usepackage{amsmath,amsthm}
\usepackage[ruled, lined, linesnumbered]{algorithm2e}
\usepackage{latexsym}
\usepackage{enumerate}
\usepackage{geometry}   %ÉèÖÃÒ³±ß¾àµÄºê°ü
\usepackage{titlesec}   %ÉèÖÃÒ³Ã¼Ò³½ÅµÄºê°ü
\usepackage{amssymb}

\usepackage{graphicx}
\usepackage[natural]{xcolor}
\theoremstyle{plain}
\newtheorem{thm}{Theorem}[section]
\newtheorem{conj}[thm]{Conjecture}
\newtheorem{claim}[thm]{Claim}

\numberwithin{theorem}{section}

\numberwithin{lemma}{section}
\newtheorem{corollary}{Corollary}
\numberwithin{corollary}{section}

\numberwithin{conjecture}{section}
\newtheorem*{conjecture*}{Conjecture}

\theoremstyle{definition}

\theoremstyle{remark}

\theoremstyle{plain}

\theoremstyle{plain}

\theoremstyle{plain}

\title{A note on antimagic orientations of even regular graphs\\}
\author{Donglei Yang\thanks{ Email address:  dlyang120@163.com (D. Yang)} \\
{ Department of Mathematics}\\
{ Shandong  University, Jinan,  China}\\
}
\date{}

\begin{document}
\maketitle
\begin{abstract}
%A \dfn{antimagic labeling} of a digraph $D$ with $m$ arcs is a bijection from the set of arcs of $D$ to $\{1, \ldots, m\}$, such that no two vertices have the same vertex-sum,  where the vertex-sum of a vertex $u\in V(D)$ is the sum of labels of all arcs entering $u$ minus the sum of labels of all arcs leaving $u$
\noindent Motivated by the conjecture of Hartsfield and Ringel
on antimagic labelings of undirected graphs, Hefetz, M\"{u}tze, and Schwartz
initiated the study of antimagic labelings of digraphs in 2010. Very recently, it has been conjectured in [Antimagic orientation of even regular graphs, J. Graph Theory, 90 (2019), 46-53.] that every graph admits an antimagtic orientation, which is a strengthening of an earlier conjecture of Hefetz, M\"{u}tze and Schwartz.
In this paper, we prove that every $2d$-regular graph (not necessarily connected) admits an antimagic orientation, where $d\ge2$. Together with known results, our main result implies that the above-mentioned conjecture is true for all regular graphs.
\end{abstract}
%\bigskip

\baselineskip 18pt
\section{Introduction}
All graphs in this paper are finite and simple.
For a graph $G$, we use $e(G)$ to denote the number of edges of $G$ and $d_G(u)$ to denote the degree of a vertex $u\in V(G)$. For a positive integer $k$, we denote $[k]:=\{1,2,\dots, k\}$. An \emph{antimagic labeling} of a graph $G$ is a bijection from $E(G)$ to $[e(G)]$ such that for any distinct vertices $u$ and $v$, the sum of labels on edges incident to $u$ differs from that for edges incident to $v$. A graph $G$ is antimagic if it has an antimagic labeling. Hartsfield and Ringel~\cite{NG} introduced antimagic labelings in 1990 and conjectured that every connected graph other than $K_2$  is  antimagic.  The first significant progress on this problem is a result of Alon, Kaplan, Lev, Roditty, and Yuster~\cite{Alon}, which states that there exists an absolute constant  $c$ such that  every graph on $n$ vertices with minimum degree at least $c\log n$  is antimagic. Later, Eccles~\cite{E} improved this by showing that there exists an absolute constant $c_0$ such that if $G$ is a graph with average
degree at least $c_0$, and $G$ contains no isolated edge and at most one isolated vertex, then $G$ is
antimagic. For some special classes of graphs, Hefetz~\cite{DH} proved that any graph on $n=3^k$ vertices that admits a triangle factor is antimagic. Later, Hefetz, Saluz and Tran \cite{DHA} generalized this by proving that any graph on $n=p^k$ vertices that admits a $C_p$-factor is antimagic, where $p$ is an odd prime.
%Cranston~\cite{DC} proved that any $d$-regular bipartite graph with   $d\ge2$ is antimagic.
For regular graphs, Cranston, Liang,  and Zhu~\cite{DYZ} proved that every odd regular graph is antimagic. Later, Chang, Liang, Pan, and Zhu~\cite{CLPZ} proved the even case.  For more result, we recommend a survey \cite{JAG}.

Motivated by antimagic labelings of graphs, Hefetz, M\"{u}tze, and Schwartz~\cite{DTJ} initiated the study of antimagic labelings of digraphs.  Let $D$ be a digraph and $A(D)$ and $V(D)$ be the set of arcs and vertices of $D$, respectively. An \emph{antimagic labeling} of $D$ is bijection from $A(D)$ to $\{1,2,\ldots,|A(D)|\}$ such that no two vertices in $D$ have the same vertex-sum,  where the vertex-sum of a vertex $u\in V(D)$ is the sum of labels of all arcs entering $u$ minus the sum of labels of all arcs leaving $u$.  A digraph is \emph{antimagic} if it has an antimagic labeling. A graph $G$ has an \emph{antimagic orientation} if an orientation of  $G$ is antimagic.

Hefetz, M\"{u}tze, and Schwartz~\cite{DTJ} raised the questions
``Is every orientation of any connected graph antimagic?'' and ``Does every graph admit an antimagic orientation?''. They proved an analogous result of Alon, Kaplan, Lev, Roditty, and Yuster~\cite{Alon} that there exists an absolute constant $c$ such that every orientation of any  graph on $n$ vertices with minimum degree at least $c\log n$ is antimagic. Except for $K_{1,2}$ and $K_3$, no other counterexample to the first question is known. %They also showed that every orientation of   star $S_n$ with  $n\ne 2$ is antimagic; every orientation of  wheel $W_n$ is antimagic; and every orientation of $K_n$ with  $n\ne3$ is antimagic.
For the second question, the same authors proposed the following conjecture.

\begin{conj}\emph{\cite{DTJ}}\label{conj1}
Every connected graph admits an antimagic orientation.
\end{conj}

They proved the following result on certain classes of regular graphs.
\begin{thm}\emph{\cite{DTJ}}\label{odd} For any integer $d\ge1$,
 \begin{description}
   \item[(a)] every $(2d-1)$-regular graph admits an antimagic orientation;
   \item[(b)] every connected, $2d$-regular graph $G$ admits an antimagic orientation if  $G$ has a matching that covers all but at most one vertex of $G$.
 \end{description}
\end{thm}

 They \cite{DTJ} also pointed out that
 ``It seems hard to discard any of the two conditions in Theorem~\ref{odd} (b), that is connectedness and having a matching that covers all but at most one vertex. In fact, we do not even know if every disjoint union of cycles admits an antimagic orientation.''

Recently, Li et al. \cite{lt} confirmed the case when $G$ is a disjoint union of cycles and also proved a stronger result than Theorem \ref{odd} (b). %A component is \emph{odd} if it has odd number of vertices.
\begin{thm}\emph{\cite{lt}}\label{li} For any integer $d\ge2$,
\begin{description}
  \item[(a)] every $2$-regular graph admits an antimagic orientation;
  \item[(b)] every $2d$-regular graph with at most two odd components admits an antimagic orientation.
\end{description}
\end{thm}

%Recently, Shan and Yu~\cite{SY} affirmed Conjecture~\ref{conj1} for biregular bipartite graphs.
In the same paper \cite{lt}, the authors proposed a stronger conjecture as follows.
\begin{conj}\emph{\cite{lt}}\label{conj2}
Every graph admits an antimagic orientation.
\end{conj}
In this paper, we prove the following main result.
\begin{thm}\label{main}
  Every $2d$-regular graph admits an antimagic orientation, where $d\geq2$.
\end{thm}
 %Actually our basic technique is similar to the one used in \cite{lt}, and the point is, we come up with new ideas to deal with the odd components uniformly.
Together with Theorem \ref{odd} (a) and Theorem \ref{li} (a), Theorem \ref{main} implies that Conjecture \ref{conj2} holds for all regular graphs.
\begin{corollary}\label{reg}
  Every regular graph admits an antimagic orientation.
\end{corollary}
Recently,  Shan and Yu ~\cite{SY} affirmed Conjecture~\ref{conj2} for biregular bipartite graphs, where we only require that any two vertices in the same part have the same degree.

It is worth noting that the proof of Thereom \ref{main} is similar to the proof of Theorem \ref{li} given in \cite{lt}. We label the edges of odd components judiciously so that we obtain a desired antimagic orientation.

We need to introduce one more notation. A closed walk in a graph $G$ is an \emph{Euler tour} if it traverses every edge of the graph exactly once.  %A graph  is \dfn{Eulerian} if it admits an Euler tour.
The following is a result of Euler.
\begin{thm}[Euler 1736]\label{euler}
A connected graph admits an Euler tour if and only if every vertex has even degree.
\end{thm}

\section{Proof of Theorem \ref{main}}
 Let $G$ be a $2d$-regular graph on $n$ vertices with $d\geq2$, then by Theorem \ref{li} (b), we assume that $G$ contains more than two odd components. Let $G_1,G_2,\ldots,G_s,G_{s+1},\ldots,G_t$ be components of $G$ with $|V(G_i)|=n_i$ and $e(G_i)=m_i$ for each $i\in[t]$, where $G_1,G_2,\ldots,G_s$ are all the odd components of $G$ with $n_1\leq n_2\leq\ldots\leq n_s$ and $s\geq3$. By Theorem \ref{euler}, each $G_i$ has an Euler tour $T_{i}$ for $i\in[t]$. We can regard each $T_i$ as a simple cycle of size $e(G_i)$, say $C^*_i$, and each vertex $u\in V(G)$ has $d$ copies in $C^*_i$. For each vertex $u\in V(G_i)$, we can arbitrarily pick one of the $d$ copies of $u$ as a real vertex and the remaining $d-1$ copies as imaginary vertices. Therefore, $C^*_i$ has $n_i$ real vertices and $(d-1)n_i$ imaginary vertices. Moreover, we can pick the real vertices as follows.
\begin{claim}\label{c1}
For any $i\in [t]$, we can pick $n_i$ real vertices in $C^*_i$ such that $C^*_i$ contain four consecutive vertices, say, $a,b,c,d$ such that $a,b,d$ are real vertices and $c$ is a imaginary vertex.
\end{claim}
\begin{proof} Since $G_i$ is $2d$-regular with $d\geq2$, we have $n_i\geq5$ and let $x_1,x_2,x_3,x_4,x_5$ be five consecutive vertices of $C^*_i$. If $x_1$ and $x_4$ are distinct vertices, then treat $x_3$ as an imaginary vertex (this is possible when $d\geq2$) and $a=x_1, b=x_2, c=x_3,d=x_4$ will suffice. If $x_1=x_4$, i.e., $x_1x_2x_3$ forms a triangle in $G_i$, then treat $x_4$ as an imaginary vertex and $a=x_2,b=x_3,c=x_4,d=x_5$ will suffice.
\end{proof}
%Denote by $V^i=\{v^i_{1},v^i_{2},\ldots,v^i_{n_i}\}$ the vertex set of $G_i$ for each $i\in [t]$.
For simplicity, let $V^i_R$ be the set of $n_i$ real vertices in $C^*_i$, and label them with $v^i_{1},v^i_{2},\ldots,v^i_{n_i}$ in order. By renaming them if necessary, together with Claim \ref{c1}, we may further assume that $a=v^i_{n_i}, b=v^i_{1},d=v^i_{2}$  for each $i\in[t]$. For each $i\in [t]$ and each $1\leq l\leq n_i-1$, let $P^i_l$ be the path between $v^i_{l}$ and $v^i_{l+1}$ in $C^*_i$ such that either $v^i_{l}v^i_{l+1}\in E(C^*_i)$ or all the internal vertices of $P^i_l$ are imaginary vertices, and we call such paths $P^i_l$ \emph{good}.

We observe that each orientation (labelling) of $C^*_i$ corresponds to an unique orientation (labelling) of $G_i$, so in the following proof we only consider orientations (labellings) on $C^*_i$.
Next, we will finish the proof by orienting and labelling the edge set of all $C^*_i$ such that the corresponding orientation of $G$ has an antimagic labelling.

We orient each $C^*_i$ as follows.
For each $C^*_i$ with $i\in[s]$, firstly we orient the edge $v^i_{n_i}v^i_{1}$ from $v^i_{n_i}$ to $v^i_{1}$, and then orient the path $P^i_1$ to be a directed path from $v^i_{1}$ to $v^i_{2}$. Next, for all $l\in\{2,3,\ldots,n_i-1\}$, we orient the path $P^i_l$ to be a directed path between $v_l^i$ and $v^i_{l+1}$ such that each real vertex in $\{v^i_{2},\ldots,v^i_{n_i}\}$ has outdegree either $2$ or $0$ (see Figure \ref{f1}).
For each $C^*_j$ with $j>s$, firstly we orient the edge $v^j_{n_j}v^j_{1}$ from $v^j_{n_j}$ to $v^j_{1}$. Next, for all $l\in\{1,2,3,\ldots,n_j-1\}$, we orient the path $P^j_l$ to be a directed path between $v_l^j$ and $v^j_{l+1}$ such that any real vertex in $\{v^j_{1},v^j_{2},\ldots,v^j_{n_j}\}$ has outdegree either $2$ or $0$ (this is possible since $n_j$ is even).

Now we have the desired orientation of all $C^*_i$, say $D$. Let $N_i=\sum\limits_{j=1}^im_j$ for each $i\in [t]$, we label $C^*_{1},C^*_{2},\ldots,C^*_{s}$ with $[N_s]$ and then $C^*_{s+1},\ldots,C^*_{t}$ with $\{N_s+1,N_s+2,\ldots,N_t\}$ as follows.

\begin{algorithm}[H]\label{label1}
\KwData{$C^*_1,C^*_2,\ldots,C^*_t$ with a given orientation $D$}
\KwResult{A bijection  $c: A(D)\rightarrow [dn]$}

\For{$i=1$ to $s$ }{
 Label $v^i_{n_i}v^i_{1}$ with $i$ and
 assign the two numbers $s+2i-1, s+2i$ to the \\two arcs in increasing order along the orientation of $P^i_1$\;
}
Denote by $L$ the set of edges which have been labelled as above \;
 \While{$L\neq \bigcup\limits_{i=1}^sA(C^*_i)$}{
{- Find an edge $e\in L$ according to the following order of preference:
\begin{enumerate}
  \item $e$ is adjacent to an edge which is unlabelled;
  \item the label on $e$ is minimal among all the edges which satisfy 1.
\end{enumerate}
Let $P$ be the good path which is unlabelled and shares an endpoint with $e$\;
}
{- Assign the $|A(P)|$ smallest unused numbers in $[N_s]$ to the arcs of $P$ in \\ increasing order along the orientation of $P$ \;
}
{
- Set $L=L\cup A(P)$ \;
}
}
\For{$i=s+1$ to $t$ }{
Label $v^i_{n_i}v^i_{1}$ with $N_{i-1}+1$ and assign the two numbers $N_{i-1}+2, N_{i-1}+3$ to \\the two arcs in increasing order along the orientation of $P^i_1$\;
Set $B=\{v^i_{n_i}v^i_{1}\}\cup A(P^i_{1})$ \;
  \While{$B\neq A(C^*_i)$}{
{- Find an arc $e\in B$ according to the following order of preference:
\begin{enumerate}
  \item $e$ is adjacent to an edge which is unlabelled;
  \item the label on $e$ is minimal among all the edges which satisfy 1.
\end{enumerate}
Let $P$ be the good path which is unlabelled and shares an endpoint with $e$\;
}
{- Assign the $|A(P)|$ smallest unused numbers in $\{N_s+1,N_s+2,\ldots,N_t\}$ to the arcs of $P$ in increasing order along the orientation of $P$ \;
}
{
- Set $B=B\cup A(P)$ \;
}}
}
\caption{Label the arcs of $D$}
\end{algorithm}
\bigskip
Note that in \textbf{Algorithm 1}, \textbf{steps} $1$-$10$ are devoted to labelling $C^*_1,C^*_2,\ldots,C^*_s$ with integers in $[N_s]$. Since $D$ is a family of oriented cycles, we say a real vertex $u$ \emph{sees} $(x,y)$ in the resulting labelling $c$ if the two arcs incident with $u$ have labels $x$ and $y$ with $x<y$. For simplicity, we say each real vertex $v^i_l$ in $C^*_i$ sees $(x^i_l,y^i_l)$ for any $i\in [s]$ and $l\in [n_i]$.

By running \textbf{steps} $7$-$8$, intuitively, we have the following claim. Using induction on $l$, we can easily prove the claim, here we omit it.
\begin{claim}\label{c3}
 For any two $ C^*_i, C^*_j$ with $i<j\leq s$ and any integer $l$ with $2\leq l\leq\lceil n_i/2\rceil$, we have $x^i_{n_i-l+2}<x^j_{n_i-l+2}<x^i_l<x^j_l$.%<y^i_{n_i-l+2}<y^j_{n_i-l+2}<y^i_l<y^j_l$ (see Figure \ref{f1}).
\end{claim}

%We claim that the consequent labelling $c$ of $G$ from \textbf{Algorithm} \ref{label1} is antimagic.
\begin{figure}[h]
  \centering
  \includegraphics[width=465pt]{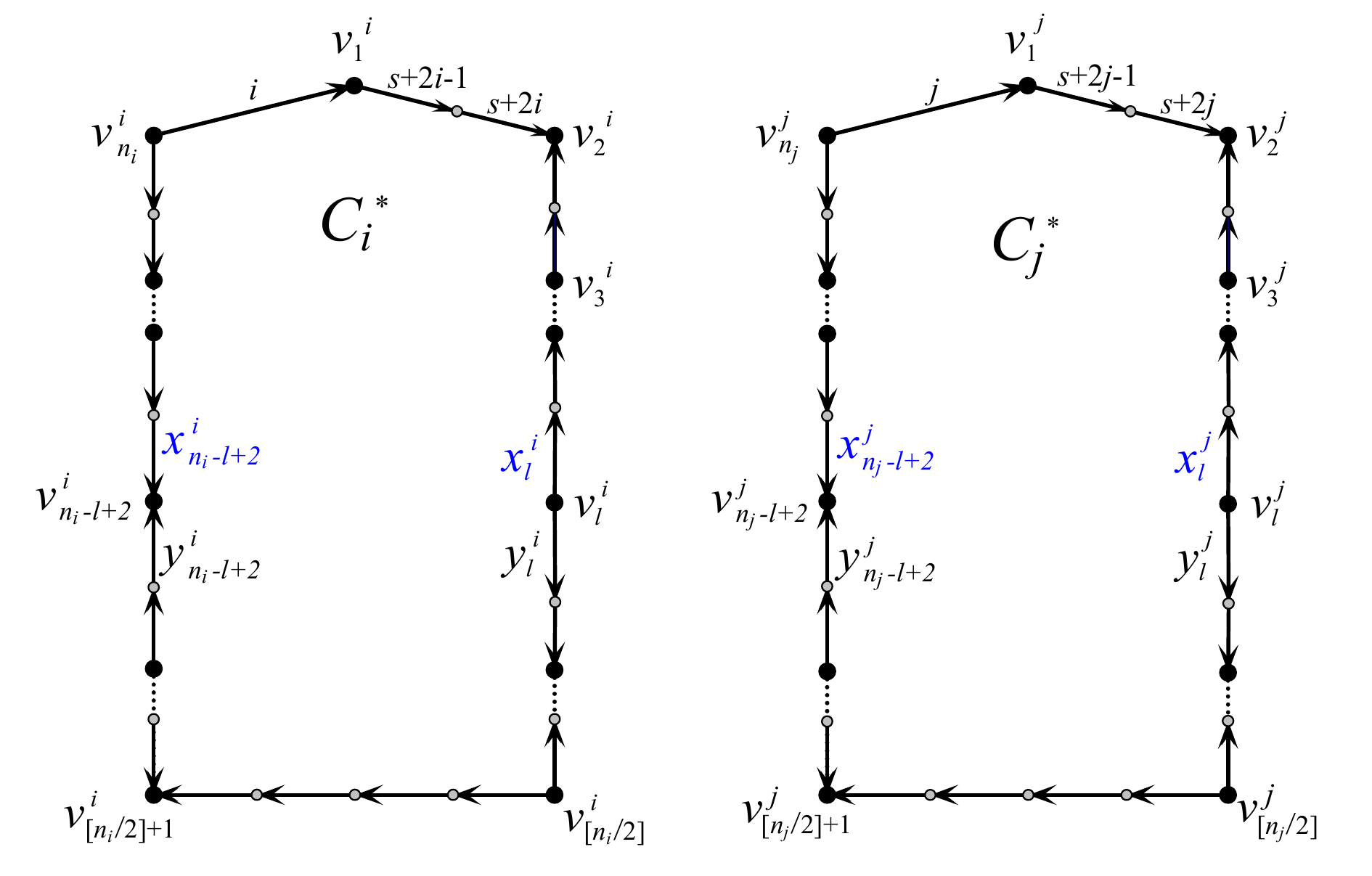}\\
  \caption{$C^*_i$ and $C^*_j$ in Claim \ref{c3}. }\label{f1}
\end{figure}
Since each $G_i$ is $2d$-regular, every vertex $u$ has $d-1$ imaginary copies in $D$ and each imaginary copy contributes $-1$ to the vertex sum of $u$ in the resulting labelling $c$. So in order to make sure that no two vertices in $G$ have the same vertex sum, it suffices to prove that no two real vertices in $D$ have the same vertex sum. For any vertex $u\in V(D)$, denote by $S(u)$ the vertex sum of $u$ in $D$. We complete our proof with the following claim.

\begin{claim}\label{c2}
  No two real vertices in $D$ have the same vertex sum.
\end{claim}
\begin{proof}
Recall that $V_R^i$ is the set of real vertices in $C^*_i$ and $|V_R^i|=n_i$ for each $i\in [t]$. Let $X=\{v^i_1| i\in [s]\}$ and $Y=\bigcup\limits^s_{i=1}V_R^i-X$. Firstly, we have the following intuitive observations.
%\textbf{Observations}
\begin{enumerate}
  \item  For each $v^i_1\in X$, $S(v^i_1)=-i-s+1$.
  \item  For each $u\in Y$, $|S(u)|\geq 3s+1$.
  \item  For any real vertices $u\in V_R^i, v\in V_R^j$ and $w\in V_R^k$ with $i<s+1\leq j<k\leq t$, we have $|S(u)|<|S(v)|<|S(w)|$.
\end{enumerate}

\medskip
Actually, it is not difficult to prove that for any real vertices $u, v\in V_R^k$ with $k>s$, we have $S(u)\neq S(v)$, and the proof can be found in \cite{lt}. Hence, it remains to prove that no two real vertices in $Y$ have the same vertex sum.

Suppose to the contrary there exist two real vertices $u,v\in Y$ with the same vertex sum and assume that $u$ sees $(a,b)$, $v$ sees $(c,d)$ with $a+b=c+d$.
Remember that $u,v$ must have the same outdegree $2$ or $0$.

If $u,v\in V_R^i$ for some $i\in [s]$, then $u,v$ cannot be two consecutive real vertices in $C^*_i$. By symmetry, we assume $a<c$, then by \textbf{steps} $6-7$, the number $b$ will be used prior to $d$, thus $b<d$, a contradiction since $a+b=c+d$.

Suppose $u\in V_R^i, v\in V_R^j$ with $i<j\leq s$, we consider the following cases depending on the locations of $u,v$.

If $u\neq v^i_{\lceil n_i/2\rceil}$ and $v\neq v^j_{\lceil n_j/2\rceil}$, then $a<c$ (or $c<a$) implies that $b$ ($d$) is used prior to $d$ ($b$), thus $b<d$ ($d<b$), a contradiction.

If $u=v^i_{\lceil n_i/2\rceil}$ and $v=v^j_{\lceil n_j/2\rceil}$, then we claim that $n_i=n_j$, in fact, $n_i<n_j$ implies $\lceil n_i/2\rceil<\lceil n_j/2\rceil$ and it follows that $a,b$ are used prior to $b,d$, i.e., $a<b<c<d$, a contradiction.
Since $i<j$ and $n_i=n_j$, we have $a<c$ and $x^i_{\lceil n_i/2\rceil+1}<x^j_{\lceil n_j/2\rceil+1}$ by Claim \ref{c3}. Thus $b$ is used in \textbf{step} $7$ prior to $d$, i.e., $b<d$, a contradiction.

If $u=v^i_{\lceil n_i/2\rceil}$ and $v=v^j_{l}$ with $l\neq\lceil n_j/2\rceil$, then we claim that $l=\lceil n_i/2\rceil$. To see this, by Claim \ref{c3}, we observe that either $l=\lceil n_i/2\rceil$ or $l=n_j-\lceil n_i/2\rceil+2$ and since $v^j_{n_j-\lceil n_i/2\rceil+2}$ and $v^i_{\lceil n_i/2\rceil}$ do not have the same outdegree in $D$, the case $l=n_j-\lceil n_i/2\rceil+2$ cannot happen.
So we have $x^i_{\lceil n_i/2\rceil+1}<x^i_{\lceil n_i/2\rceil}=a<c$, and it follows that $b$ is used in \textbf{step} $7$ prior to $d$, i.e., $b<d$, a contradiction.

If $u=v^i_{l}$ with $l\neq\lceil n_i/2\rceil$ and $v=v^j_{\lceil n_j/2\rceil}$, then by Claim \ref{c3}, we must have $n_i=n_j$ and $l=\lceil n_i/2\rceil+1$. By the orientation $D$, we can easily observe that $u,v$ do not have the same out-degree, a contradiction. This completes the proof of Claim \ref{c2}.
\end{proof}

%Claim \ref{c2} implies that $c$ is an antimagic labelling.
%\hfill\vrule height3pt width6pt depth2pt \medskip

\noindent \textbf{Remark.} As we present an orientation and an antimagic labelling explicitly, the proof of Theorem \ref{main} is elementary. Claim \ref{c1} turns out to be very helpful, that is, we could pick three real vertices and control the length of path between any two consecutive real vertices to some extent.
Applying the same arguments in the proof of Theorem \ref{main}, we can prove a more general result.
\begin{corollary}
Let $d\ge2$ be an integer. If every vertex of a graph $G$  has degree $2d$ or $2d+2$,  then $G$ admits an antimagic orientation.
\end{corollary}
%The proof of Theorem \ref{mian} is elementary, the result comprises further progress towards proving the conjecture of Hefetz, M\"utze and Schwartz
%It seems not easy to prove that $G$ admits an antimagic orientation on condition that each component has an antimagic orientation. With the support of Theorem~\ref{reg}, we believe the following  is true.
%\begin{conj}[\cite{lt}]
%Every graph admits an antimagic orientation.
%\end{conj}

\noindent \textbf{Acknowledgement.}
The author would like to thank Tong Li, Zi-Xia Song and Guanghui Wang for helpful discussion. This work was supported by the National Natural Science Foundation of China (11871311).
%%%%%%%%%
%%%%%%%%%
%%%%%%%%%
%%%%%%%%%

%\newpage

\end{document}